\documentclass[]{fairmeta}

\usepackage{amsmath, amssymb, amsthm}
\usepackage{algorithm}
\usepackage{algpseudocode}
\usepackage{graphicx}

\newtheorem{proposition}{Proposition}

\newtheorem{remark}{Remark}

\DeclareMathOperator*{\argmin}{argmin}
\DeclareMathOperator*{\argmax}{argmax}
\DeclareMathOperator{\prox}{prox}
\DeclareMathOperator{\proj}{\Pi}

\title{Solving Minimax Problems with Bilinear Objectives with ADMM}

\author[1,*]{Bob Wilson}

\affiliation[1]{Marketing Data Science at Meta}

\contribution[*]{Work done at Meta}

\date{\today}
\correspondence{Bob Wilson at \email{rwilson4@meta.com}}

\abstract{We consider minimax (saddle-point) problems of the form
  $\max_{c \in \mathcal{C}} \min_{\beta \in S} g(c; \beta)$, where
  $\mathcal{C}$ and $S$ are compact convex sets, and $g$ is
  concave-convex. Applying the Alternating Direction Method of
  Multipliers (ADMM) requires evaluating a proximal operator that is,
  in general, as hard as the original problem. We show that when the
  outcome function $g$ is bilinear, i.e. $g(c; \beta) = c^T A \beta,$
  the proximal operator reduces to a \emph{generalized projection}
  onto the confidence region $S$. This reduction is exact --- it
  involves no approximation or linearization. The resulting ADMM
  algorithm alternates between (i)~a generalized projection onto $S$
  and (ii)~a Euclidean projection onto $\mathcal{C}.$ We describe the
  derivation, state the algorithm, and discuss convergence.}

\begin{document}
\maketitle

\section{Introduction}
Data science analyses are often conducted to inform some decision. The
decision seeks to achieve some outcome (or perhaps maximize some
objective) by choosing from some set of options we'll call the
\emph{decision space}. The decision space may be some finite set of
options, such as the color of a button or the decision to ship or not
ship some feature. Or the decision space may be a continuum, such as
the choice of how to allocate marketing resources across multiple
marketing channels. This paper focuses on the latter type, where the
decision space is some compact set,
$ \mathcal{C} \subseteq \mathbb{R}^n. $

If the decision maker could anticipate the consequences of their
decision on the outcome, say in the form of a deterministic function
mapping decisions to outcomes, decision making would be an exercise in
mathematical optimization. Often this connection involves unknowns,
and the data scientist contributes by estimating relevant quantities.

Let $ g(c; \beta) : \mathcal{C} \times \mathcal{B} \to \mathbb{R} $
represent the outcome associated with the decision $ c. $ Without loss
of generality, we assume larger values correspond to better outcomes.
The outcome is parametrized by one or more quantities, $ \beta, $
which are known a priori to belong to some support $ \mathcal{B}
\subseteq \mathbb{R}^m. $ For example, $ \beta $ might represent
conversion rates associated with $ m $ options, and thus we know $
\beta \in [0, 1]^m. $

The analysis delivers a point estimate, $ \hat{\beta} \in \mathcal{B}$
-- a best guess as to the value of $ \beta. $ Armed with this estimate
we might calculate
\begin{equation}\label{eqn:naive_optimal}
  c_0 := \argmax_{c \in \mathcal{C}} {g(c;  \hat{\beta})}
\end{equation}
as our recommended decision. Yet this strategy neglects the
uncertainty associated with $ \hat{\beta}. $ No matter how well the
analysis was performed, some uncertainty is inescapable. This
uncertainty may be quantified with a \emph{confidence set}, $ S
\subseteq \mathcal{B}. $ While it is implausible that $ \beta =
\hat{\beta} $ exactly, it is comparatively defensible that $ \beta \in
S. $ Thus it is $ S, $ more so than $ \hat{\beta}, $ that represents
what the analysis has taught us about $ \beta. $

If the outcome associated with a decision, $ c, $ varies dramatically
across the confidence set, we cannot be confident in the consequences
of our decision. Uncertainty in $ \beta $ leads to uncertainty in the
anticipated outcome, and poses a kind of risk to the decision. This
risk we quantify as the infimum of $ g $ over $ \beta \in S: $
\begin{equation}
  f(c) := \inf_{\beta \in S} g(c; \beta).
\end{equation}
While $ g(\cdot; \hat{\beta}) $ calculates the expected outcome,
$ f(\cdot) $ calculates the worst-case outcome (conditional on
$ \beta \in S$).

The decision $ c_0, $ made on the basis of the point estimate $
\hat{\beta}, $ may involve considerable risk if $f(c_0)$ is materially
lower than $ g(c_0; \hat{\beta}). $ For this reason, it is na\"ive to
attribute any notion of optimality to $ c_0; $ henceforth, we call it
the \emph{na\"ive optimal} decision.

Another decision may involve materially less risk (having a better
worst-case outcome), with perhaps only a slightly worse expected
outcome. To assess this possibility, we solve:
\begin{equation}\label{eqn:dual}
  \begin{array}{ll}
    \mbox{maximize}_c & f(c) \\
    \mbox{subject to} & c \in \mathcal{C},
  \end{array}
\end{equation}
which is equivalent to the maximin problem:
\begin{displaymath}
  \begin{array}{ll}
    \mbox{maximize}_c \; \mbox{minimize}_\beta & g(c; \beta) \\
    \mbox{subject to} & \beta \in S \\
                      & c \in \mathcal{C}.
  \end{array}
\end{displaymath}
Call the solution to~(\ref{eqn:dual}) the \emph{robust optimal}
decision.

\section{Allocating Marketing Resources Across Channels}

As an example, consider allocating marketing resources across
$ n $ digital channels. Given a maximum allocation $ B, $ the decision
space is $ \mathcal{C} = \{ c : c \succeq 0, \sum_i c_i \leq B \}, $
where $ c \succeq 0 $ enforces nonnegative allocation to each channel.

Many digital channels permit randomized \emph{lift studies}, where an
audience is divided into \emph{Marketing} and \emph{Holdout} groups;
people in the Holdout group do not see ads. The \emph{uplift} for a
channel is the difference in conversion rates, Marketing minus
Holdout:
\begin{displaymath}
  \xi_i = \beta_i^\mathrm{M} - \beta_i^{\mathrm{H}}.
\end{displaymath}
Channels may differ in how resource-intensive it is to reach people. If we
allocate $ c_i $ to channel $ i $, and it requires $ \rho_i $ to reach one
person, a simple model leads us to expect
\begin{align}
  g(c; \beta) &:= c_i \xi_i / \rho_i \nonumber \\
  &= c_i \left( \beta_i^\mathrm{M} - \beta_i^{ \mathrm{H}} \right) / \rho_i \nonumber \\
  &= c^T A \beta \label{eqn:outcome_model}
\end{align}
incremental outcomes, where
\begin{displaymath}
  A = \begin{pmatrix}
    -1/\rho_1 & 1/\rho_1 & 0 & 0 & \cdots \\
    0 & 0 & -1/\rho_2 & 1/\rho_2 & \cdots \\
    \vdots & & & & \ddots
  \end{pmatrix},
\end{displaymath}
and $ \beta = [\beta_1^\mathrm{H}, \beta_1^\mathrm{M}, \dots] \in [0, 1]^{2n} =: \mathcal{B}. $

This simple model neglects the \emph{diminishing returns} property of
marketing: often as we increase allocation to a channel, we see reduced
effectiveness. Yet we may view Equation~(\ref{eqn:outcome_model}) as
an approximation to the true relationship, by viewing it as a Taylor
expansion around the baseline allocation profile used in the test. We may
view the model as reliable only for decisions in a neighborhood or
\emph{trust region} of this baseline. We can incorporate this
restriction into the definition of the decision space,
$ \mathcal{C}; $ we don't consider it further here.

Randomized lift studies deliver point estimates of uplift in each
channel. The na\"ive optimal decision allocates the entire allocation to
whichever channel has the best $ \xi / \rho $ ratio, or perhaps
maximizing contribution to this channel over the trust region where
the linear relationship is reliable. Yet this strategy involves risk
unless we estimate $\xi$ precisely, which is rare in practice.

A confidence region, $ S, $ quantifies the uncertainty associated with
finite sample sizes. With large sample sizes and conversion rates not
too close to 0 or 1, we can approximate $ S $ as ellipsoidal, $
S_\mathrm{ellipsoid} = \{(\beta - \hat{\beta})^T P (\beta -
\hat{\beta}) \leq 1\}. $ It is often better to use a region based on
the likelihood ratio test applied to the binomial log-likelihood:
\begin{displaymath}
  S_\mathrm{binomial} = \{ \beta \in [0,1]^{2n} : 2 \cdot (\ell(\hat{\beta}) - \ell(\beta)) \leq  \chi^2_{1-\alpha,2n} \},
\end{displaymath}
where $\ell(\beta) = \sum_i \bigl[ s_i \log \beta_i + (t_i - s_i)
\log(1 - \beta_i) \bigr]$ is the binomial log-likelihood, $s_i$ and
$t_i$ are the observed successes and trials, and $\hat{\beta}_i := s_i
/ t_i$ is the maximum likelihood estimate. When channels have
conversion rates close to 0 (as is often the case), $
S_\mathrm{ellipsoid} $ may include points outside of the support $
\mathcal{B}; $ $ S_\mathrm{binomial} $ is always $ \subseteq [0,
1]^{2n}. $ The ellipsoidal approximation admits a closed-form
Markowitz solution (\S\ref{sec:markowitz}), but the exact
likelihood-ratio region requires an iterative solver such as the ADMM
method developed here.

Solving~(\ref{eqn:dual}) is not always tractable, but when
$ g(c; \beta) $ is concave in $ c $ for each $\beta,$ $ f(c) $ is
concave (being the infimum of a family of concave functions
\cite[\S3.2.3]{boyd2004convex}), and~(\ref{eqn:dual}) is tractable.
Moreover, when $ g(c; \beta) $ is convex in $ \beta $ for each $c$ and
$ S $ is a convex set, evaluating $f(c)$ involves solving a convex
optimization problem.

In this case, a saddle point condition characterizes the solution.
There exist $c^\star$ and $ \beta^\star$ such that $ g(c; \beta^\star)
\leq g(c^\star; \beta^\star) \leq g(c^\star; \beta)$ for all $ c \in
\mathcal{C} $ and all $ \beta \in S. $ The robust optimal decision
will be $ c^\star, $ having a better outcome than any other decision
when $ \beta = \beta^\star, $ and $ \beta^\star $ corresponds to the
worst value of the parameters, having a worse outcome than any other
in the confidence set.

Several approaches to finding such a saddle point are available.
Projected subgradient methods make minimal assumptions but converge
slowly. Accelerated proximal gradient methods are faster but require $
f(c) $ to be differentiable. This holds only when the corresponding
minimizer is unique. This may not apply if $ g $ is not strongly
convex in $ \beta. $ The Markowitz approach yields a closed-form
quadratic program when $g$ is bilinear and $S$ is ellipsoidal, but
does not extend to other confidence regions.

In this paper, we develop an ADMM-based solver that occupies a useful
middle ground. It converges reliably, makes no smoothness assumptions
on $ f, $ and applies to any convex confidence region $S.$ The key
insight is: when $g$ is a \emph{matrix game}, i.e., $g(c; \beta) = c^T
A \beta,$ we compute the proximal operator arising in ADMM exactly via
a \emph{generalized projection} onto $S.$

\section{ADMM for the Dual Problem}

\subsection{Reformulation}

We rewrite~\eqref{eqn:dual} as a minimization:
\begin{equation}\label{eq:min-form}
  \begin{array}{ll}
    \mbox{minimize} & \tilde{f}(c) + I_{\mathcal{C}}(c),
  \end{array}
\end{equation}
where $\tilde{f}(c) := -f(c) = \sup_{\beta \in S} \{-g(c; \beta)\}$
and $I_{\mathcal{C}}$ is the indicator function of $\mathcal{C}$:
$I_{\mathcal{C}}(c) = 0$ if $c \in \mathcal{C}$ and $+\infty$
otherwise.

To apply ADMM, we introduce an auxiliary variable $y$ and
reformulate~\eqref{eq:min-form} as the consensus problem
\begin{displaymath}
  \begin{array}{ll}
    \mbox{minimize} & \tilde{f}(y) + I_{\mathcal{C}}(c) \\
    \mbox{subject to} & y = c.
  \end{array}
\end{displaymath}
The augmented Lagrangian (in scaled form, following
\citealt{boyd2011distributed}, \S3.1.1) is
\begin{displaymath}
  L_\rho(y, c, u) = \tilde{f}(y) + I_{\mathcal{C}}(c)
  + \frac{\rho}{2} \|y - c + u\|_2^2,
\end{displaymath}
where $u$ is the scaled dual variable and $\rho > 0$ is the penalty
parameter.

\subsection{ADMM Updates}

The ADMM iteration consists of three steps:
\begin{align}
  y^{k+1} &= \prox_{\tilde{f}/\rho}\!\bigl(c^k - u^k\bigr),
  \label{eq:y-update} \\
  c^{k+1} &= \prox_{I_{\mathcal{C}}/\rho}\!\bigl(y^{k+1} + u^k\bigr)
           = \proj_{\mathcal{C}}\!\bigl(y^{k+1} + u^k\bigr),
  \nonumber \\
  u^{k+1} &= u^k + y^{k+1} - c^{k+1}. \nonumber
\end{align}

The $c$-update is simply Euclidean projection onto $\mathcal{C}$, a
convex optimization problem. The $u$-update is trivial. The challenge
lies in the $y$-update~\eqref{eq:y-update}, which requires evaluating
the proximal operator of $\tilde{f}$.

\section{The Proximal Reduction}

\subsection{The Proximal Operator of $\tilde{f}$}

Expanding the definition:
\begin{equation}\label{eq:prox-def}
  \prox_{\tilde{f}/\rho}(v) = \argmin_y \left\{
    \sup_{\beta \in S} \bigl\{-g(y; \beta)\bigr\}
    + \frac{\rho}{2} \|y - v\|_2^2
  \right\},
\end{equation}
where we write $v := c^k - u^k$ for brevity. This is a minimax
problem: minimize over $y$, maximize over $\beta \in S$. In general,
evaluating this proximal operator is as hard as the original problem.

When $g(y; \beta) = y^T A \beta$,
\eqref{eq:prox-def} becomes
\begin{displaymath}
  \begin{array}{ll}
    \mbox{minimize}_y \; \mbox{maximize}_\beta
    & -y^T A \beta + \frac{\rho}{2} \|y - v\|_2^2 \\
    \mbox{subject to} & \beta \in S.
  \end{array}
\end{displaymath}
The inner function is linear (hence concave) in $\beta$ and strongly
convex in $y$ (due to the quadratic penalty). By the minimax theorem
\citep{sion1958}, we may interchange the order of optimization:

\begin{displaymath}
  \min_y \max_{\beta \in S} \left\{
    -y^T A \beta + \frac{\rho}{2} \|y - v\|_2^2
  \right\}
  = \max_{\beta \in S} \min_y \left\{
    -y^T A \beta + \frac{\rho}{2} \|y - v\|_2^2
  \right\}.
\end{displaymath}

After the interchange, we first minimize over $y$. The first-order
condition is
\begin{displaymath}
  -A\beta + \rho(y - v) = 0
  \quad \Longrightarrow \quad
  y^\star(\beta) = v + \frac{1}{\rho} A\beta.
\end{displaymath}
Substituting $y^\star(\beta)$ back into the objective:
\begin{align}
  &-\bigl(v + \tfrac{1}{\rho}A\beta\bigr)^T A\beta
  + \frac{\rho}{2} \bigl\|\tfrac{1}{\rho}A\beta\bigr\|_2^2 \notag \\
  &\quad= -v^T A\beta - \frac{1}{\rho}\|A\beta\|_2^2
  + \frac{1}{2\rho}\|A\beta\|_2^2 \notag \\
  &\quad= -v^T A\beta - \frac{1}{2\rho}\|A\beta\|_2^2 \notag \\
  &\quad= -\frac{1}{2\rho}\bigl\|A\beta + \rho v\bigr\|_2^2
  + \frac{\rho}{2}\|v\|_2^2.
  \label{eq:reduced}
\end{align}
Since $\frac{\rho}{2}\|v\|_2^2$ is constant with respect to $\beta$,
maximizing~\eqref{eq:reduced} over $\beta \in S$ is equivalent to
\begin{displaymath}
  \min_{\beta \in S} \; \|A\beta + \rho v\|_2^2
  = \min_{\beta \in S} \; \|A\beta - (-\rho v)\|_2^2.
\end{displaymath}

When $g(c; \beta) = c^T A\beta$, the proximal operator
$\prox_{\tilde{f}/\rho}(v)$ can be computed as follows:
\begin{enumerate}
  \item Perform a \emph{generalized projection} of $ -\rho v $ onto S:
    \begin{displaymath}
      \beta^\star = \argmin_{\beta \in S} \|A\beta - (-\rho v)\|_2^2.
    \end{displaymath}
  \item Set $y^\star = v + \frac{1}{\rho}A\beta^\star$.
\end{enumerate}
Then $y^\star = \prox_{\tilde{f}/\rho}(v)$.

\begin{remark}
  A \emph{generalized projection} of a point $w$ onto a set $S$
  through a linear map $A$ is
  $\argmin_{\beta \in S} \|A\beta - w\|_2^2$. When $A = I$, this
  reduces to Euclidean projection. The generalized projection is a
  convex optimization problem (minimizing a convex quadratic over a
  convex set).
\end{remark}

\section{The Complete Algorithm}

Combining the proximal reduction with the standard ADMM updates, we
obtain Algorithm~\ref{alg:admm}. Algorithm~\ref{alg:admm} involves
alternating generalized projections onto $S$
(line~\ref{line:gen-proj}) followed by Euclidean projections onto
$\mathcal{C}$ (line~\ref{line:c-update}).

\begin{algorithm}[t]
\caption{ADMM for Minimax Matrix Games}
\label{alg:admm}
\begin{algorithmic}[1]
\Require Outcome matrix $A$, confidence region $S$, decision space
$\mathcal{C}$, penalty $\rho > 0$, tolerances $\epsilon_{\text{abs}}$,
$\epsilon_{\text{rel}}$.
\State Initialize $c^0 \in \mathcal{C}$, $\beta^0 \in S$,
  $u^0 = 0$.
\For{$k = 0, 1, 2, \ldots$}
  \State Set $v^k \leftarrow c^k - u^k$.
  \State \textbf{Generalized projection:}
    $\beta^{k+1} \leftarrow \displaystyle\argmin_{\beta \in S}
    \|A\beta + \rho v^k\|_2^2$.
    \label{line:gen-proj}
  \State \textbf{Auxiliary update:}
    $y^{k+1} \leftarrow v^k + \frac{1}{\rho} A\beta^{k+1}$.
    \label{line:y-update}
  \State \textbf{Decision projection:}
    $c^{k+1} \leftarrow \proj_{\mathcal{C}}(y^{k+1} + u^k)$.
    \label{line:c-update}
  \State \textbf{Dual update:}
    $u^{k+1} \leftarrow u^k + y^{k+1} - c^{k+1}$.
    \label{line:u-update}
  \State \textbf{Primal residual:}
    $r^{k+1} \leftarrow y^{k+1} - c^{k+1}$.
  \State \textbf{Dual residual:}
    $s^{k+1} \leftarrow \rho(c^{k+1} - c^k)$.
  \State Compute tolerances:
  \[
    \epsilon_{\text{pri}} = \sqrt{n}\,\epsilon_{\text{abs}}
    + \epsilon_{\text{rel}} \max\bigl(\|y^{k+1}\|_2,
    \|c^{k+1}\|_2\bigr),
  \]
  \[
    \epsilon_{\text{dual}} = \sqrt{n}\,\epsilon_{\text{abs}}
    + \epsilon_{\text{rel}}\,\rho\,\|u^{k+1}\|_2.
  \]
  \If{$\|r^{k+1}\|_2 \leq \epsilon_{\text{pri}}$ and
    $\|s^{k+1}\|_2 \leq \epsilon_{\text{dual}}$}
    \State \textbf{break}
  \EndIf
\EndFor
\State Recover worst-case parameters:
  $\beta^\star \leftarrow
  \argmin_{\beta \in S} g(c^k; \beta)$.
\State \Return $(c^k, \beta^\star)$.
\end{algorithmic}
\end{algorithm}

\section{Convergence}

The convergence of ADMM for convex problems is well established
\citep{boyd2011distributed, eckstein1992douglas}. Our formulation is a
standard two-block ADMM problem with convex objectives $\tilde{f}$ and
$I_{\mathcal{C}}$, so the classical convergence guarantees apply
directly.

\begin{proposition}[Convergence]
  Let $\tilde{f}$ and $I_{\mathcal{C}}$ be closed, proper, convex
  functions, and suppose
  the unaugmented Lagrangian $L_0$ has a saddle point. Then the
  iterates of Algorithm~\ref{alg:admm} satisfy:
  \begin{enumerate}
    \item \textbf{Primal feasibility:}
      $\|r^k\|_2 \to 0$ as $k \to \infty$.
    \item \textbf{Dual feasibility:}
      $\|s^k\|_2 \to 0$ as $k \to \infty$.
    \item \textbf{Objective convergence:}
      $\tilde{f}(y^k) + I_{\mathcal{C}}(c^k) \to p^\star$, the optimal
      value.
  \end{enumerate}
\end{proposition}

\begin{proof}
  This is a direct application of \citet[Theorem~1]{eckstein1992douglas}
  (see also \citealt[Appendix~A]{boyd2011distributed}), since both
  $\tilde{f}$ and $I_{\mathcal{C}}$ are closed, proper, and convex.
\end{proof}

The convergence criteria used in Algorithm~\ref{alg:admm} follow
\citet[\S3.3.1]{boyd2011distributed}. The primal residual $r^{k+1} =
y^{k+1} - c^{k+1}$ measures consensus violation. The dual residual
$s^{k+1} = \rho(c^{k+1} - c^k)$ measures stationarity of the
Lagrangian. The tolerances combine absolute and relative components
scaled by the problem dimension.

\begin{remark}[Convergence rate]
  Standard ADMM has an $O(1/k)$ convergence rate for the objective
  value \citep{he2012convergence}. In practice, ADMM often converges
  faster than this worst case, particularly when we solve the
  proximal subproblems exactly, as we do in our setting.
\end{remark}

\section{Comparison with Other Approaches}

\subsection{Projected Subgradient Methods}

Projected subgradient methods make essentially no assumptions beyond
evaluability and projectability. However, they converge at rate
$O(1/\sqrt{k})$ and require a diminishing step size sequence, making
them slow in practice.

\subsection{Accelerated Proximal Gradient Methods}

Accelerated Proximal Gradient (APG) methods achieve $O(1/k^2)$
convergence when the objective is smooth. Applied to the dual problem,
APG requires the dual objective $f(c)$ to be differentiable. By
Danskin's theorem, $f$ is differentiable when the inner minimizer
$\beta^\star(c) = \argmin_{\beta \in S} g(c; \beta)$ is unique.
Uniqueness holds when $g$ is strictly convex in $\beta$, but for a
matrix game $g(c; \beta) = c^T A\beta$ is merely linear in $\beta$, so
uniqueness depends on the geometry of $S$. If $S$ has flat faces (as
likelihood-ratio confidence regions for binomial proportions can), the
dual objective may be nondifferentiable, causing the APG line search
to fail.

In practice, APG may fail to converge when the inner
minimization over $ \beta $ is solved only to moderate precision (as
is typical of interior point methods). In that case, the gradient
given by Danskin's theorem is only approximate, and this can prevent
APG from converging. However, ADMM may still
converge in these scenarios.

\subsection{Direct Quadratic Programming (Markowitz)}
\label{sec:markowitz}

When $S$ is an ellipsoid $\{(\beta - \hat{\beta})^T P (\beta -
\hat{\beta}) \leq 1\}$ and $\mathcal{C}$ is a simplex, the dual
objective has a closed form: $f(c) = c^T A\hat{\beta} - \|P^{-1/2} A^T
c\|_2$. The problem reduces to a second-order cone program, solvable
by standard methods. This is the fastest approach when it applies, but
it applies only to ellipsoidal confidence regions.

ADMM applies to \emph{any} convex confidence region, including
likelihood-ratio regions for binomial data.

\subsection{Summary}

\begin{center}
\begin{tabular}{llll}
  \hline
  \textbf{Method} & \textbf{Rate} & \textbf{Smoothness?}
  & \textbf{Conf.\ region} \\
  \hline
  Subgradient & $O(1/\sqrt{k})$ & No & Any \\
  APG & $O(1/k^2)$ & Yes & Any \\
  Markowitz & 1 iteration & N/A & Ellipsoidal \\
  \textbf{ADMM} & $O(1/k)$ & \textbf{No} & \textbf{Any}$^\dagger$ \\
  \hline
  \multicolumn{4}{l}{\small $^\dagger$Requires generalized projection
    onto $S.$}
\end{tabular}
\end{center}

\section{Discussion}

\subsection{The Role of the Matrix Game Assumption}

The proximal reduction depends critically on the bilinearity of $g$.
For a general outcome function, the proximal
operator~\eqref{eq:prox-def} involves a minimax problem that cannot be
decomposed as cleanly. The bilinear structure ensures that (i) the
minimax interchange is valid (the regularized objective is
convex--concave), (ii) the inner minimization over $y$ has a
closed-form solution, and (iii) the resulting outer maximization over
$\beta$ reduces to a standard convex problem.

\subsection{Choice of $\rho$}

The penalty parameter $\rho$ affects convergence speed but not the
converged solution. Large $\rho$ penalizes consensus violation and
reduces the primal residual, but may slow dual convergence; small
$\rho$ has the opposite effect. In practice, $\rho = 1$ works well for
problems where the decision and parameter spaces have comparable
scales. Adaptive schemes for varying $\rho$ across iterations appear
in \citet[\S3.4.1]{boyd2011distributed}.

\subsection{Final Recovery of $\beta^\star$}

The $\beta^{k+1}$ iterates produced by the generalized projection step
are not the worst-case parameters for $c^k$. They are intermediate
quantities in the proximal computation. After ADMM converges, we
recover the true worst-case parameters by evaluating the dual
objective at the final $c^k$: $\beta^\star = \argmin_{\beta \in S}
g(c^k; \beta)$.

\subsection{Warm-Starting}

We may wish to solve not just one problem of the
form~(\ref{eqn:dual}), but a family of problems:
\begin{equation}\label{eqn:pareto}
  \begin{array}{ll}
    \mbox{maximize}_c \; \mbox{minimize}_\beta & g(c; \beta) \\
    \mbox{subject to} & \beta \in S \\
                      & c \in \mathcal{C} \\
                      & g(c; \hat{\beta}) \geq \varphi,
  \end{array}
\end{equation}
varying $ \varphi $ to trade off between worst-case and expected
performance. Since the last constraint involves only $c$, we treat it
as part of the decision space and is thus covered under the notation
used throughout this paper. We interpret this as considering only
decisions that have a certain minimum expected outcome. Evaluating
this tradeoff curve involves solving several minimax problems.

Algorithm~\ref{alg:admm} accepts initial guesses $c^0$ and $\beta^0$.
When solving the family~(\ref{eqn:pareto}), we start from the largest
feasible value $ \varphi = h(\hat{\beta}), $ where
\begin{equation}
  h(\beta) = \sup_{c \in \mathcal{C}} g(c; \beta),
\end{equation}
giving $c^\star = c_0.$ This solution (and the corresponding $
\beta^\star $) remain feasible if we reduce $ \varphi, $ and thus
serve as $ c^0 $ and $\beta^0$ in the next problem. Warm-starting from
the earlier solution can substantially reduce the number of iterations
needed.

\section{Numerical Experiment}

We illustrate the convergence behavior of ADMM on a simulated
lift-study problem with $n = 5$ channels. The parameter vector $\beta
\in [0,1]^{10}$ has holdout and marketing conversion rates for each
channel, and the $5 \times 10$ outcome matrix $A$ has the structure
described in \S2.2. The confidence region $S$ is the acceptance region
of a likelihood-ratio test for binomial proportions, with each group
having between 200 and 500 trials. The decision space $\mathcal{C}$ is
a simplex. In both cases, custom interior point methods
handle the projections (generalized onto $S,$ Euclidean onto
$\mathcal{C}$).

Figure~\ref{fig:convergence} compares three solvers: ADMM
(Algorithm~\ref{alg:admm}), accelerated proximal gradient (APG) ascent
on the dual objective, and projected subgradient ascent. The left
panel shows the ADMM primal and dual residuals. The right panel plots
the duality gap $h(\beta^k) - f(c^k)$ for all three solvers on a
common vertical axis.

The residual trajectories exhibit a characteristic staircase pattern:
the dual residual decreases in discrete steps, and at each step the
primal residual spikes before recovering. This pattern reflects the
combinatorial structure of the simplex. The projection $c^{k+1}
= \Pi_{\mathcal{C}}(y^{k+1} + u^k)$ lands on a face of the polytope
determined by an \emph{active set:} the subset of channels receiving
zero allocation. While the active set remains stable, $c$ moves only
slightly within the face, so the dual residual $\rho\|c^{k+1} -
c^k\|_2$ is small and flat. Meanwhile, the scaled dual variable $u$
steadily accumulates the consensus violation $y - c.$ Eventually $u$
grows large enough that $y^{k+1} + u^k$ crosses a face boundary: a
channel gains or loses its allocation, the active set changes, and $c$
jumps to a new face. The dual residual spikes briefly, then settles to
a lower plateau because the new face is closer to the solution. The
primal residual $\|y^{k+1} - c^{k+1}\|_2$ spikes simultaneously
because $y^{k+1}$ was computed from the old $c^k - u^k$ and has not
yet adapted to the new face; it recovers on the next iteration. In
effect, ADMM is solving a combinatorial subproblem (identifying which
channels belong in the optimal support) through a sequence of
continuous relaxations, and the staircase reveals each revision of
that support.

APG converges fastest, reaching a duality gap below $10^{-5}$ within
50 iterations. ADMM converges to a gap below
$10^{-3}$ within 200 iterations, suggesting ADMM is most suitable for
problems where moderate precision is acceptable. Projected subgradient
ascent, despite its theoretical convergence guarantee, makes only slow
progress, with the gap hovering near $10^{-2}$ after 200 iterations,
consistent with its $O(1/\sqrt{k})$ rate.

\begin{figure}[t]
  \centering
  \includegraphics[width=\textwidth]{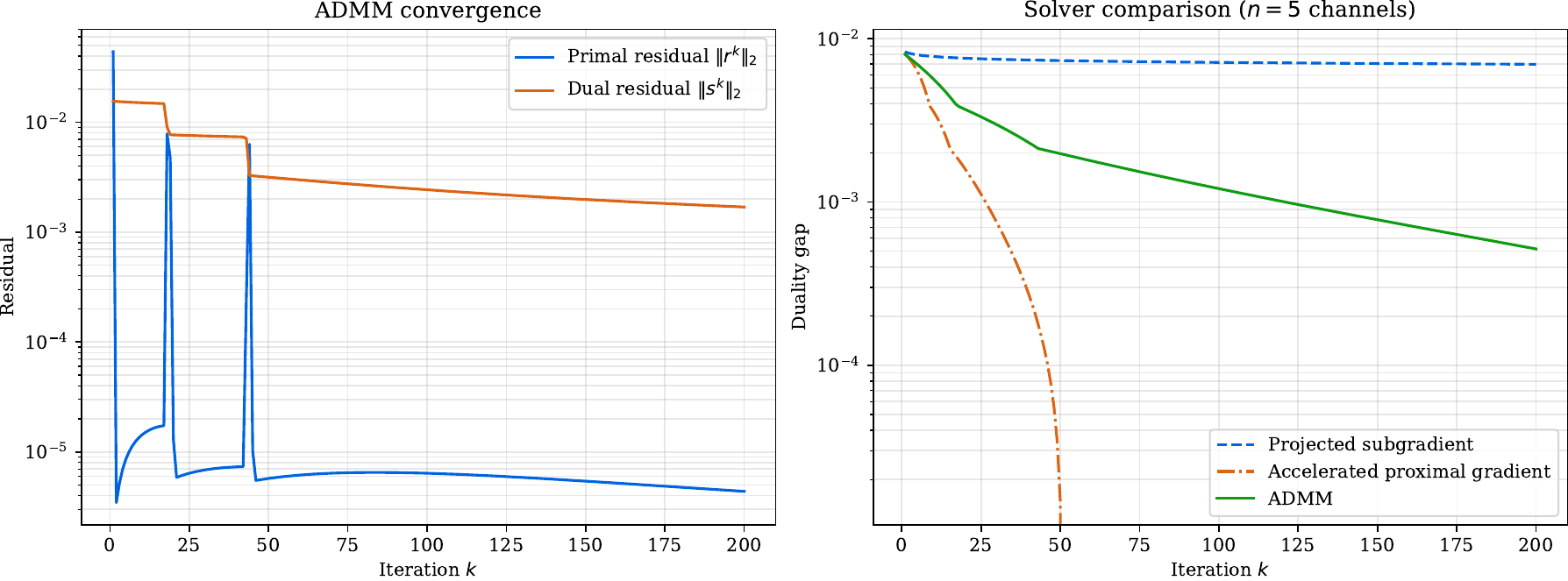}
  \caption{Convergence comparison on a simulated lift study with
    $n = 5$ channels and a binomial likelihood-ratio confidence
    region. \emph{Left:} ADMM primal and dual residuals (log
    scale). \emph{Right:} Duality gap $h(\beta) - f(c)$ for ADMM, APG, and
    projected subgradient ascent (log scale).}
  \label{fig:convergence}
\end{figure}

\section{Conclusion}

We have shown ADMM applies to minimax problems of the form $\max_{c
\in \mathcal{C}} \min_{\beta \in S} c^T A \beta$ by reducing the
proximal operator to a generalized projection onto the confidence
region. The resulting algorithm is conceptually simple,
computationally efficient, and numerically stable.

The approach occupies a useful niche: it is not as fast as the
Accelerated Proximal Gradient method, but its convergence is reliable
under broader conditions. ADMM
provides a reliable solver for the resource allocation problems
motivating this work.

\bibliographystyle{plainnat}
\bibliography{admm_minimax}

\begin{thebibliography}{5}
\providecommand{\natexlab}[1]{#1}
\providecommand{\url}[1]{\texttt{#1}}
\expandafter\ifx\csname urlstyle\endcsname\relax
  \providecommand{\doi}[1]{doi: #1}\else
  \providecommand{\doi}{doi: \begingroup \urlstyle{rm}\Url}\fi

\bibitem[Boyd and Vandenberghe(2004)]{boyd2004convex}
Stephen Boyd and Lieven Vandenberghe.
\newblock \emph{Convex Optimization}.
\newblock Cambridge University Press, 2004.

\bibitem[Boyd et~al.(2011)Boyd, Parikh, Chu, Peleato, and
  Eckstein]{boyd2011distributed}
Stephen Boyd, Neal Parikh, Eric Chu, Borja Peleato, and Jonathan Eckstein.
\newblock Distributed optimization and statistical learning via the alternating
  direction method of multipliers.
\newblock \emph{Foundations and Trends in Machine Learning}, 3\penalty0
  (1):\penalty0 1--122, 2011.

\bibitem[Eckstein and Bertsekas(1992)]{eckstein1992douglas}
Jonathan Eckstein and Dimitri~P. Bertsekas.
\newblock On the {D}ouglas--{R}achford splitting method and the proximal point
  algorithm for maximal monotone operators.
\newblock \emph{Mathematical Programming}, 55\penalty0 (1):\penalty0 293--318,
  1992.

\bibitem[He and Yuan(2012)]{he2012convergence}
Bingsheng He and Xiaoming Yuan.
\newblock On the $o(1/n)$ convergence rate of the {D}ouglas--{R}achford
  alternating direction method.
\newblock \emph{SIAM Journal on Numerical Analysis}, 50\penalty0 (2):\penalty0
  700--709, 2012.

\bibitem[Sion(1958)]{sion1958}
Maurice Sion.
\newblock On general minimax theorems.
\newblock \emph{Pacific Journal of Mathematics}, 8\penalty0 (1):\penalty0
  171--176, 1958.

\end{thebibliography}

\end{document}